\documentclass[11pt]{article}%
\usepackage{amsmath}
\usepackage{amsfonts}
\usepackage{amssymb}
\usepackage{graphicx}%
\newtheorem{theorem}{Theorem}

\newtheorem{corollary}[theorem]{Corollary}

\newtheorem{definition}[theorem]{Definition}
\newtheorem{example}[theorem]{Example}

\newtheorem{lemma}[theorem]{Lemma}

\newtheorem{remark}[theorem]{Remark}

\newenvironment{proof}[1][Proof]{\noindent\textbf{#1.} }{\ \rule{0.5em}{0.5em}}
\begin{document}

\begin{center}
{\LARGE On the Factorization of Nonlinear Recurrences in Modules}

\medskip

H. SEDAGHAT \footnote{Department of Mathematics, Virginia Commonwealth
University, Richmond, VA \ 23284, USA;
\par
Email: hsedagha@vcu.edu}
\end{center}

\medskip

\begin{abstract}
For rings $R$ with identity, we define a class of nonlinear higher order
recurrences on unitary left $R$-modules that include linear recurrences as
special cases. We obtain conditions under which a recurrence of order $k+1$ in
this class is equivalent to a pair, known as a semiconjugate factorization,
that consists of a recurrence of order $k$ and a recurrence of order 1. We
show that such a factorization is possible whenever $R$ contains certain
sequences of units. Further, if the coefficients of the original recurrence in
$R$ are independent of the index then we show that the semiconjugate
factorization exists if two characteristic polynomials share a common root
that is a unit in $R$. We use this fact to show that an overlapping
factorization of these polynomials in an integral domain $R$ yields a
semiconjugate factorization of the corresponding recurrence in the module.
These results are applicable to systems of higher order, nonlinear difference
equations in direct products of rings. Such systems may be represented as
higher order equations in a module over the ring.

\end{abstract}

\medskip



\section{Introduction}

Let $R$ be a ring with identity, $M$ a (unitary) left $R$-module and $k$ a
non-negative integer. Then for any given sequence of maps $f_{n}%
:M^{k+1}\rightarrow M$ the difference equation
\begin{equation}
x_{n+1}=f_{n}(x_{n},x_{n-1},\ldots,x_{n-k}) \label{de}%
\end{equation}
defines a recurrence of order $k+1$ in $M$; see Section \ref{horm} below for a
more precise definition. If each $f_{n}$ is linear then the recurrence is
linear, otherwise it is \textit{nonlinear}.

Methods of linear algebra are fruitfully used in the factorization of
\textit{linear} recurrences. These methods apply generally to linear systems
in commutative rings with identity. Such systems include standard unfoldings
of higher order linear recurrences in a ring $R$ to \textit{first-order}
recurrences in an $R$-module over the ring. These notions generalize the
familiar concept that a difference (or differential) equation of order $k$ can
be unfolded to a first-order equation in the (real) vector space of dimension
$k$. First-order linear recurrences in modules and algebras with coefficients
in rings are studied using standard methods; see, e.g., \cite{abu},
\cite{bir}, \cite{bro}, \cite{hen}, \cite{kur}, \cite{lak}, \cite{sedR},
\cite{sen}. In particular, the classical operator factorization of difference
(and differential) equations follows from these methods.

By contrast, no standard algebraic methods are known for the factorization of
\textit{nonlinear} difference equations, even in concrete cases involving
familiar fields such as the real or complex numbers. A method known as
\textit{semiconjugate factorization} (Section \ref{scfg} below) applies to
large classes of nonlinear recurrences as well as to linear ones. For the
linear cases, the factorization obtained is the familiar operator
factorization from the classical theory of linear difference equations. It
coincides with what is obtained from linear algebra as noted above.

More generally, semiconjugate factorization applies when a \textit{form
symmetry} of the given recurrence is identified according to an existing
classification scheme that is not limited to linear equations. The form
symmetry is used as an order-reducing substitution to break down, or factor,
the original equation into two equations of lower orders. This factorization
may be repeated as long as form symmetries are identified. For some background
on semiconjugate factorization we refer to \cite{fsor}.

In this paper we study a form of semiconjugate factorization that works well
for nonlinear recurrences of type%
\begin{equation}
x_{n+1}=\sum_{i=0}^{k}a_{i,n}x_{n-i}+g_{n}\left(  \sum_{i=0}^{k}b_{i,n}%
x_{n-i}\right)  \label{mn}%
\end{equation}
where the coefficients $a_{i,n},b_{i,n}$ are in a ring $R$ with identity. For
each $n$ the map $g_{n}:M\rightarrow M$ is defined on a (unitary) left
$R$-module $M$ that also contains the variable $x_{n}$ (see the next section
for more precise definitions). Since each ring $R$ is a left $R$-module over
itself, the ideas and results in this paper also apply to recurrences in
rings, where $x_{n}\in R$ for all $n.$

When $M=R$, (\ref{mn}) generalizes linear recurrences since the latter are
represented by (\ref{mn}) with $g_{n}(u)=0$ or $g_{n}(u)=c_{n}u$ for all $u\in
R$ and a given sequence of constants $c_{n}\in R$. \textit{Non-homogeneous}
linear recurrences are also special cases of (\ref{mn}) where $g_{n}(u)=d_{n}$
for every $u\in R$ and a given sequence of constants $d_{n}\in R$.

Special cases of Equation (\ref{mn}) on the set of real numbers
($M=R=\mathbb{R}$) have often appeared in the applied literature. The
classical economic models of the business cycle in mid-twentieth century are
among the early occurrences; see, e.g. \cite{Hic}, \cite{Puu}, \cite{Sam}.
Other special cases of (\ref{mn}) occurred later in mathematical studies of
biological models ranging from whale populations to neuron activity; see,
e.g., \cite{Clr}, \cite{FG}, \cite{Ham} and Section 2.5 in \cite{KL}. For
instance, the global dynamics of the solutions of the following equation are
discussed in \cite{Ham}
\[
x_{n+1}=\alpha x_{n}+a\tanh\left(  x_{n}-\sum_{i=1}^{k}b_{i}x_{n-i}\right)
\]
with constant real parameters $0\leq\alpha<1$, $a>0$ and $b_{i}\geq0$. The
studies of global dynamics for other special cases of (\ref{mn}) appear in
\cite{GLV} and \cite{KPS}; also see \cite{KL}, Section 6.9.

The dynamical properties of the solutions of the second-order case
\begin{equation}
x_{n+1}=cx_{n}+g(x_{n}-x_{n-1}) \label{sed}%
\end{equation}
have been studied in \cite{Elm}, \cite{KS}, \cite{S97}. Further, the
bifurcations of solutions of (\ref{sed}), including the Neimark-Sacker type
(discrete analog of the Hopf bifurcation) are studied in \cite{LZ}. A more
general form of (\ref{sed}), i.e.
\begin{equation}
x_{n+1}=ax_{n}+bx_{n}+g_{n}(x_{n}-cx_{n-1}) \label{sed1}%
\end{equation}
is studied in \cite{SKy}. In particular, \cite{dkmos} presents sufficient
conditions on parameters for the occurrence of limit cycles (attracting
periodic solutions) and chaos in recurrences of type%
\[
x_{n+1}=\frac{ax_{n}^{2}+bx_{n-1}^{2}+cx_{n}x_{n-1}+dx_{n}+ex_{n-1}+f}{\alpha
x_{n}+\beta x_{n-1}+\gamma}%
\]

This equation is obtained by using a single rational mapping $g$ in
(\ref{sed1}) of type
\[
g(r)=\frac{Ar+B}{Cr+D}.
\]

Like linear non-homogeneous equations mentioned above, equations such as
(\ref{mn}) are also meaningful in more general algebraic contexts such as
rings or modules. Further, In \cite{sedb} the autonomous version of (\ref{mn})
where the coefficients do not depend on $n$ and $g_{n}=g$ is fixed for all
$n$, is studied for normed algebras over real or complex numbers. Refinements
of this study in the case of complex coefficients are discussed in
\cite{sedC}. On the other hand, in \cite{sedR}\ we study the linear,
non-homogeneous special case of (\ref{mn}) with variable coefficients in
rings, along with some examples and applications. Although these studies are
largely focused on the dynamics of solutions, the method that is used to
reduce (\ref{mn}) to simpler equations is algebraic in nature. We focus on
this aspect and discuss the method in the more general setting of left modules
on rings with identity.

The layout of this paper is as follows: In Section \ref{horm} we discuss some
background issues pertaining to higher order nonlinear recurrences in modules
and in Section \ref{scfg} we present the basics of form symmetries and
semiconjugate factorization of higher order recurrences. Next, we define the
linear form symmetry in Section \ref{main} and obtain the corresponding
semiconjugate factorization of (\ref{mn}) over left $R$-modules where $R$ is a
ring with identity. In Section \ref{constc} we study the case of constant
coefficients and establish a connection between semiconjugate factorization of
recurrences and polynomial factorization in rings. Extending the above ideas
to vector spaces and modules allows the consideration of systems of nonlinear,
higher order difference equations in direct products of rings. We clarify this
issue in Section \ref{sys}. In Section \ref{repfac} we discuss repeated
semiconjugate factorizations and conditions for the complete factorization of
a higher order recurrence into a system of first-order recurrences. Section
\ref{red} starts a discussion of reducibility of (\ref{mn}) in a more general
context than polynomial factorization, in particular, the possibility of
obtaining a semiconjugate factorization where polynomial factorization is not
possible in the underlying ring. This section is open-ended and mainly a
starting point for possible future research.

\section{Higher order recurrences in modules\label{horm}}

In this section we discuss the basics of recurrences in modules. The basics of
rings and modules may be found in texts such as \cite{hun}\ and \cite{mlb}.
For a nonempty set $M$ let $S=M^{\mathbb{N}}$ be the set of all sequences in
$M$; here $\mathbb{N}$ is the set of all positive integers. If $R$ is a ring
with identity and $M$ is a left (unitary) $R$-module then $S$ is also a left
$R$-module under the usual operations of term-wise addition of sequences and
multiplying by scalars (elements of $R$).

For all $\{x\}=\{x_{1},x_{2},\ldots,x_{n},\ldots\}\in S$ and each
$n\in\mathbb{N}$ we define a projection map $\pi_{n}:S\rightarrow M$ as
$\pi_{n}\{x\}=x_{n}$ where $x_{n}$ is the $n$-th term of $\{x\}.$ Note that if
$\{x\},\{y\}\in S$ and $\pi_{n}\{x\}=\pi_{n}\{y\}$ for all $n\in\mathbb{N}$
then $\{x\}=\{y\}.$ Further, if $M$ is a left $R$-module then $\pi_{n}$ is a
left $R$-module epimorphism on $S$ for every $n.$

Next, the \textit{(forward) shift map} $E:S\rightarrow S$ is defined as
$\pi_{n}\circ E=\pi_{n+1}$ for all $n\in\mathbb{N}$; i.e.
\[
\pi_{n}\circ E\{x\}=\pi_{n+1}\{x\}\quad\text{for }\{x\}\in S\text{ and }%
n\in\mathbb{N}%
\]
or simply, $E\{x_{1},x_{2},\ldots\}=\{x_{2},x_{3},\ldots\}$. $E$ is
well-defined in this way since if $\{x^{\prime}\}=E\{x\}$ and also
$\{x^{\prime\prime}\}=E\{x\}$ then $\pi_{n}\{x^{\prime}\}=\pi_{n+1}%
\{x\}=\pi_{n}\{x^{\prime\prime}\}$ for all $n\in\mathbb{N}$ so $\{x^{\prime
}\}=\{x^{\prime\prime}\}.$ Further, $E$ is the unique operator on $S$ with
this property since if $E^{\prime}:S\rightarrow S$ also satisfies $\pi
_{n}\circ E^{\prime}=\pi_{n+1}$ for all $n\in\mathbb{N}$ then $E^{\prime}=E$
because for each $\{x\}\in S$ and all $n$
\[
\pi_{n}\circ E^{\prime}\{x\}=\pi_{n+1}\{x\}=\pi_{n}\circ E\{x\}.
\]

As $E$ raises each index of a sequence by 1, repeated applications of $E$
define additional shifts via $E^{2}=E\circ E,$ etc. It is readily verified
that $E$ is a left $R$-module epimorphism when $M$ is a module and the kernel
of $E$ is the set of all sequences $\{x,0,0,\ldots\}$ for $x\in M$. The set of
all fixed points of $E$ is the set of all constant sequences $\{x,x,\ldots\}$
in $S$; this set is a left $R$-submodule and a copy of $M$ in $S$.

To define higher order recurrences in left $R$-modules, let $k$ be a
non-negative integer and for each $n\in\mathbb{N}$ let $f_{n}:M^{k+1}%
\rightarrow M$ be a given map. Consider the set $\mathfrak{S}$, possibly
empty, of all sequences $\{x\}\in S$ that satisfy the equation%
\begin{equation}
\pi_{n}\circ E^{k+1}\{x\}=f_{n}\left(  \pi_{n}\circ E^{k}\{x\},\pi_{n}\circ
E^{k-1}\{x\},\ldots,\pi_{n}\circ E\{x\},\pi_{n}\{x\}\right)  \label{dee}%
\end{equation}
for every $n\in\mathbb{N}$. If $\mathfrak{S}$ is nonempty then we refer to
(\ref{dee}) as a \textit{recurrence of order} $k+1$ in $M$ and consider each
member of $\mathfrak{S}$ a \textit{solution} of (\ref{dee}). In the
first-order case where $k=0$, (\ref{dee}) reduces to the following
\begin{equation}
\pi_{n}\circ E=f_{n}\circ\pi_{n} \label{dee1}%
\end{equation}

A solution of (\ref{dee1}) is a sequence in $S$ for which equality in
(\ref{dee1}) holds for all $n\in\mathbb{N}$.

Note that the two sides of (\ref{dee}) are in $M$ rather than in $S$. To
simplify the notation, we write (\ref{dee}) in the abbreviated form%
\[
x_{n+k+1}=f_{n}(x_{n+k},x_{n+k-1},\ldots,x_{n}),\qquad n\geq0
\]
which is equivalent to (\ref{de}). The same set of initial values $x_{0}%
,x_{1},\ldots,x_{k}$ generate identical solutions in both cases, with $n\geq
k$ in (\ref{de}).

In the classical theory a \textquotedblleft scalar" recurrence of order $k+1$
in a field such as the real or complex numbers is often unfolded to a
first-order recurrence in a vector space of dimension $k+1$ in the following
way: functions $F_{n}:M^{k+1}\rightarrow M^{k+1}$ are defined as%
\[
F_{n}(u_{0},u_{1},\ldots,u_{k})=(f_{n}(u_{0},u_{1},\ldots,u_{k}),u_{0}%
,\ldots,u_{k-1})
\]
and used to recover (\ref{de}) form%
\begin{align}
(x_{0,n+1},\ldots,x_{k,n+1})  &  =F_{n}(x_{0,n},\ldots,x_{k,n})\nonumber\\
&  =(f_{n}(x_{0,n},\ldots,x_{k,n}),x_{0,n},\ldots,x_{k-1,n}) \label{veq}%
\end{align}
where $x_{j,n}=x_{n-j}$ are new variables. In this sense, (\ref{de}) in the
real or complex context is considered a special case of the \textquotedblleft
vector equation" or system of first-order equations,%
\[
(x_{0,n+1},\ldots,x_{k,n+1})=F_{n}(x_{0,n},\ldots,x_{k,n}).
\]

However, (\ref{dee}) is actually a higher-order generalization of this
first-order system which is the special case $k=0$ of (\ref{dee}) with $M$
being the vector space of dimension $k+1$. Thus, (\ref{de}) in the context of
modules extends the standard nonlinear theory to \textit{systems of higher
order equations} in direct products of rings; see comments in the Introduction
pertaining to the system (\ref{nso2a})-(\ref{nso2b}) and Example \ref{ds}
below. There is no standard theory comparable to that for first-order systems
in the existing literature for analyzing higher order nonlinear systems. The
analysis of such a system may be simplified in cases where semiconjugate
factorization reduces the order of the system to one.

\section{Semiconjugate factorization\label{scfg}}

In this section we list some general results from \cite{fsor} that are valid
for all recurrences, not only the linear ones. Let $\mathcal{G}$ be a
nontrivial group and assume in (\ref{de}) that $f_{n}:\mathcal{G}%
^{k+1}\rightarrow\mathcal{G}$. Let us unfold (\ref{de}) in the manner
described in the preceding section (which still works in this setting) to a
first-order recurrence%
\[
X_{n+1}=\mathfrak{F}_{n}(X_{n})
\]
on $\mathcal{G}^{k+1}$ where $\mathfrak{F}_{n}:\mathcal{G}^{k+1}%
\rightarrow\mathcal{G}^{k+1}$. We assume that $k\geq1$ because we shall define
a factorization of a recurrence into lower order recurrences and first-order
recurrences are already lowest in order. In analogy with polynomial
factorization, we consider first-order recurrences to be trivially irreducible.

Let $1\leq m\leq k$ and suppose that there is a sequence of maps $\Phi
_{n}:\mathcal{G}^{m}\rightarrow\mathcal{G}^{m}$ and a sequence of surjective
maps $H_{n}:\mathcal{G}^{k+1}\rightarrow\mathcal{G}^{m}$ that satisfy the
\textit{semiconjugate relation}%
\begin{equation}
H_{n+1}\circ\mathfrak{F}_{n}=\Phi_{n}\circ H_{n} \label{scr}%
\end{equation}
for a given pair of function sequences $\{\mathfrak{F}_{n}\}$ and $\{\Phi
_{n}\}.$ This may be illustarted as follows:%
\[%
\begin{array}
[c]{ccc}%
\mathcal{G}^{k+1} & \overset{F_{n}}{\longrightarrow} & F_{n}(\mathcal{G}%
^{k+1})\\
\downarrow_{H_{n}} &  & \downarrow_{H_{n+1}}\\
H_{n}(\mathcal{G}^{k+1})=\mathcal{G}^{m} & \overset{\Phi_{n}}{\longrightarrow}
& \Phi_{n}(H_{n}(\mathcal{G}^{k+1}))=H_{n+1}(F_{n}(\mathcal{G}^{k+1}))
\end{array}
\]

We say that $\mathfrak{F}_{n}$ is \textit{semiconjugate} to $\Phi_{n}$ for
each $n$ and that the sequence $\{H_{n}\}$ is a \textit{form symmetry} of
(\ref{de}). Since $m<k+1,$ the form symmetry $\{H_{n}\}$ is
\textit{order-reducing}. Note that if $H_{n}=H$ for all $n$ where $H$ is
injective ($m=k+1$) then (\ref{scr}) is a conjugacy relation between
$\{\mathfrak{F}_{n}\}$ and $\{\Phi_{n}\}$.

We state the next basic result from \cite{fsor} as a lemma here without proof.

\begin{lemma}
\label{scf}(Semiconjugate factorization) Let $(\mathcal{G},\ast)$ be a
nontrivial group\textit{ and let }$k\geq1$\textit{, }$1\leq m\leq k$\textit{
be integers. If }$h_{n}:\mathcal{G}^{k-m+1}\rightarrow\mathcal{G}$\textit{\ is
a sequence of functions }and the functions $H_{n}:\mathcal{G}^{k+1}%
\rightarrow\mathcal{G}^{m}$ are defined by\textit{\ }%
\[
H_{n}(u_{0},u_{1}\ldots,u_{k})=[u_{0}\ast h_{n}(u_{1},\ldots,u_{k-m+1}),\ldots
u_{m-1}\ast h_{n-m+1}(u_{m},\ldots,u_{k})]
\]
where $\ast$ denotes the group operation in $\mathcal{G}$ then the following
statements are true:\textit{\ }

(a) The function $H_{n}$ is surjective for every $n\geq0$.

(b) \textit{If }$\{H_{n}\}$\textit{\ is an order-reducing form symmetry then
the difference equation (\ref{de})\ is equivalent to the system of equations }%
\begin{align}
t_{n+1}  &  =\phi_{n}(t_{n},\ldots,t_{n-m+1}),\label{tdf}\\
x_{n+1}  &  =t_{n+1}\ast h_{n+1}(x_{n},\ldots,x_{n-k+m})^{-1} \label{tdcf}%
\end{align}
\textit{whose orders }$m$\textit{\ and }$k-m+1$\textit{\ respectively, add up
to the order of (\ref{de}). }

(c) The map $\Phi_{n}:\mathcal{G}^{m}\rightarrow\mathcal{G}^{m}$ is the
standard unfolding of Eq.(\ref{tdf}) for each $n\geq0$.
\end{lemma}

\begin{definition}
The pair of equations (\ref{tdf}) and (\ref{tdcf}) constitute the
\textbf{semiconjugate factorization}, or \textbf{sc-factorization} of
(\textit{\ref{de}}). This pair of equations is a triangular system since
(\ref{tdf}) is independent of (\ref{tdcf}). We call (\ref{tdf}) the
\textbf{factor }equation of (\textit{\ref{de}}) and (\ref{tdcf}) its
\textbf{cofactor }equation\textit{. }
\end{definition}

Note that (\ref{tdf}) has order $m$ and (\ref{tdcf}) has order $k-m+1.$
Consider the following special case of $H_{n}$ in Lemma \ref{scf} with $m=k$%
\begin{equation}
H_{n}(u_{0},\ldots,u_{k})=[u_{0}\ast h_{n}(u_{1}),u_{1}\ast h_{n-1}%
(u_{2}),\ldots,u_{k-1}\ast h_{n-k+1}(u_{k})] \label{fsk1}%
\end{equation}
where $h_{n}:\mathcal{G}\rightarrow\mathcal{G}$ is a given sequence of maps.
The semiconjugate factorization of (\ref{de}) in this case is%
\begin{align}
t_{n+1}  &  =\phi_{n}(t_{n},\ldots,t_{n-k+1}),\label{tdf1}\\
x_{n+1}  &  =t_{n+1}\ast h_{n+1}(x_{n})^{-1} \label{tdcf1}%
\end{align}
in which the factor equation has order $k$ and the cofactor equation has order 1.

The next result gives a necessary and sufficient condition for the existence
of a form symmetry of type (\ref{fsk1}); see \cite{fsor} for the proof.

\begin{lemma}
\label{imc}(Invertible-map criterion) Let $(\mathcal{G},\ast)$ be a nontrivial
group\textit{ and }assume that $h_{n}:\mathcal{G}\rightarrow\mathcal{G}$ is a
sequence of bijections. For arbitrary elements $u_{0},v_{1},\ldots,v_{k}%
\in\mathcal{G}$ and every $n\geq0$ define $\zeta_{0,n}(u_{0})\equiv u_{0}$ and
for $j=1,\ldots,k$ define%
\begin{equation}
\zeta_{j,n}(u_{0},v_{1},\ldots,v_{j})=h_{n-j+1}^{-1}(\zeta_{j-1,n}(u_{0}%
,v_{1},\ldots,v_{j-1})^{-1}\ast v_{j}). \label{hzetajn}%
\end{equation}

Then (\textit{\ref{de}}) has the form symmetry (\ref{fsk1}) if and only if the
quantity
\begin{equation}
f_{n}(\zeta_{0,n},\zeta_{1,n}(u_{0},v_{1}),\ldots,\zeta_{k,n}(u_{0}%
,v_{1},\ldots,v_{k}))\ast h_{n+1}(u_{0}) \label{tdhinvcrit}%
\end{equation}
is independent of $u_{0}$ for every $n\geq0$. In this case (\textit{\ref{de}})
has a semiconjugate factorization into (\ref{tdf1}) and (\ref{tdcf1}) where
the factor functions in (\ref{tdf1}) are given by%
\begin{equation}
\phi_{n}(v_{1},\ldots,v_{k})=f_{n}(\zeta_{0,n},\zeta_{1,n}(u_{0},v_{1}%
),\ldots,\zeta_{k,n}(u_{0},v_{1},\ldots,v_{k}))\ast h_{n+1}(u_{0}).
\label{tdinvf}%
\end{equation}

\end{lemma}

For a left $R$-module $M$, the group $\mathcal{G}$ in the preceding result is
the underlying abelian group $(M,+)$ so that $\ast$ denotes addition $+$ and
group inversion is the ordinary negative. Thus (\ref{hzetajn}),
(\ref{tdhinvcrit}) and (\ref{tdcf1}) read, respectively, as follows%
\begin{gather}
\zeta_{j,n}(u_{0},v_{1},\ldots,v_{j})=h_{n-j+1}^{-1}(v_{j}-\zeta_{j-1,n}%
(u_{0},v_{1},\ldots,v_{j-1})),\label{top}\\
f_{n}(\zeta_{0,n},\zeta_{1,n}(u_{0},v_{1}),\ldots,\zeta_{k,n}(u_{0}%
,v_{1},\ldots,v_{k}))+h_{n+1}(u_{0}),\label{mid}\\
x_{n+1}=t_{n+1}-h_{n+1}(x_{n}). \label{bot}%
\end{gather}

\section{sc-factorization relative to the linear form symmetry \label{main}}

To obtain a sc-factorization for the recurrence (\ref{mn}) we identify a
suitable form symmetry $H_{n}$. Unless otherwise noted, we will assume in the
rest of the paper that $R$ is a ring with identity and $M$ a left (unitary)
$R$-module.

\begin{definition}
\label{fsl}Let $\{\alpha_{n}\}$ be a sequence in $R$ such that $\alpha
_{n}\not =0$ for all $n.$ A\textit{ }\textbf{linear form symmetry} on $M$ is
defined as the special case of (\ref{fsk1}) with $h_{n}(u)=-\alpha_{n-1}u$ for
all $u\in M$ ; i.e.,
\begin{equation}
\lbrack u_{0}-\alpha_{n-1}u_{1},u_{1}-\alpha_{n-2}u_{2},\ldots,u_{k-1}%
-\alpha_{n-k}u_{k}] \label{lfs}%
\end{equation}

\end{definition}

If $\alpha$ is \textit{not} a zero divisor in $R$ then the mapping
$h(u)=-\alpha u$ is clearly injective. In general, $h$ is not surjective even
if $R$ contains no zero divisors (e.g. $\alpha\in\mathbb{Z}$, $\alpha\not =%
\pm1$). On the other hand, if each $\alpha$ is a unit in $R$ then each $h$ is
a bijection with inverse $h^{-1}(u)=-\alpha^{-1}u.$

With the linear form symmetry, Equations (\ref{top})-(\ref{bot}) read as
follows:%
\begin{gather}
\zeta_{j,n}(u_{0},v_{1},\ldots,,v_{j})=-\alpha_{n-j}^{-1}(v_{j}-\zeta
_{j-1,n}(u_{0},v_{1},\ldots,v_{j-1})),\label{t1}\\
f_{n}(\zeta_{0,n},\zeta_{1,n}(u_{0},v_{1}),\ldots,\zeta_{k,n}(u_{0}%
,v_{1},\ldots,v_{k}))-\alpha_{n}u_{0},\label{t2}\\
x_{n+1}=t_{n+1}+\alpha_{n}x_{n}. \label{t3}%
\end{gather}

The following is a straightforward consequence of Lemma \ref{imc} and proved
by simple induction.

\begin{lemma}
\label{tdlfs} Equation (\ref{de}) has the linear form symmetry (\ref{lfs}) on
the $R$-module $M$ if and only if there is a sequence $\{\alpha_{n}\}$ of
units in $R$ such that the quantity%
\begin{equation}
f_{n}(u_{0},\zeta_{1,n}(u_{0},v_{1}),\ldots,\zeta_{k,n}(u_{0},v_{1}%
,\ldots,v_{k}))-\alpha_{n}u_{0} \label{tdlfscr}%
\end{equation}
is independent of $u_{0}$ for all $n$, where for $j=1,\ldots,k$,
\begin{equation}
\zeta_{j,n}(u_{0},v_{1},\ldots,v_{j})=\left(  \prod_{i=1}^{j}\alpha
_{n-i}\right)  ^{-1}u_{0}-\sum_{i=1}^{j}\left(
{\displaystyle\prod_{l=i}^{j}}
\alpha_{n-l}\right)  ^{-1}v_{i}. \label{zen}%
\end{equation}

\end{lemma}

The following is one of the main results of this paper.

\begin{theorem}
\label{scmn}Equation (\ref{mn}) has the linear form symmetry if and only if
there is a sequence of units in $R$ that satisfies both of the following
equations:%
\begin{align}
a_{0,n}+a_{1,n}\alpha_{n-1}^{-1}+a_{2,n}\left(  \alpha_{n-1}\alpha
_{n-2}\right)  ^{-1}+\cdots+a_{k,n}\left(  \alpha_{n-1}\alpha_{n-2}%
\cdots\alpha_{n-k}\right)  ^{-1}  &  =\alpha_{n}\label{esa}\\
b_{0,n}+b_{1,n}\alpha_{n-1}^{-1}+b_{2,n}\left(  \alpha_{n-1}\alpha
_{n-2}\right)  ^{-1}+\cdots+b_{k,n}\left(  \alpha_{n-1}\alpha_{n-2}%
\cdots\alpha_{n-k}\right)  ^{-1}  &  =0. \label{esb}%
\end{align}

If such a unitary sequence say, $\{\rho_{n}\}$ exists then the
sc-factorization of (\ref{mn}) in $M$ is%
\begin{align}
t_{n+1}  &  =-\sum_{j=1}^{k}\sum_{i=1}^{j}a_{j,n}\left(  \gamma_{ij}\right)
^{-1}t_{n-i+1}+g_{n}\left(  -\sum_{j=1}^{k}\sum_{i=1}^{j}b_{j,n}\left(
\gamma_{ij}\right)  ^{-1}t_{n-i+1}\right) \label{mnf}\\
x_{n+1}  &  =\rho_{n}x_{n}+t_{n+1} \label{mncf}%
\end{align}
where $\gamma_{ij}=%
{\displaystyle\prod_{l=i}^{j}}
\rho_{n-l}$ and its inversion occurs in the group of units of $R$.
\end{theorem}

\begin{proof}
The quantity (\ref{tdlfscr}) for (\ref{mn}) is
\begin{equation}
a_{0,n}u_{0}+\sum_{j=1}^{k}a_{j,n}\zeta_{j,n}+g_{n}\left(  b_{0,n}u_{0}%
+\sum_{j=1}^{k}b_{j,n}\zeta_{j,n}\right)  -\rho_{n}u_{0} \label{aeq}%
\end{equation}
\noindent where $\zeta_{j,n}$ abbreviates $\zeta_{j,n}(u_{0},v_{1}%
,\ldots,v_{j})$. Using (\ref{zen}) we obtain%
\begin{align*}
a_{0,n}u_{0}+\sum_{j=1}^{k}a_{j,n}\zeta_{j,n}  &  =(a_{0,n}-\rho_{n}%
)u_{0}+\sum_{j=1}^{k}a_{j,n}\left[  \left(  \prod_{i=1}^{j}\rho_{n-i}\right)
^{-1}u_{0}-\sum_{i=1}^{j}\left(
{\displaystyle\prod_{l=i}^{j}}
\rho_{n-l}\right)  ^{-1}v_{i}\right] \\
&  =\left[  a_{0,n}-\rho_{n}+\sum_{j=1}^{k}a_{j,n}\left(  \prod_{i=1}^{j}%
\rho_{n-i}\right)  ^{-1}\right]  u_{0}-\sum_{j=1}^{k}\sum_{i=1}^{j}%
a_{j,n}\left(
{\displaystyle\prod_{l=i}^{j}}
\rho_{n-l}\right)  ^{-1}v_{i}%
\end{align*}

Now, setting the coefficient of $u_{0}$ equal to zero shows that $\{\rho
_{n}\}$ satisfies (\ref{esa}). Similarly, for the quantity inside $g_{n}$ we
obtain%
\begin{align*}
b_{0,n}u_{0}+\sum_{j=1}^{k}b_{j,n}\zeta_{j,n}  &  =b_{0,n}u_{0}+\sum_{j=1}%
^{k}b_{j,n}\left[  \left(  \prod_{i=1}^{j}\rho_{n-i}\right)  ^{-1}u_{0}%
-\sum_{i=1}^{j}\left(
{\displaystyle\prod_{l=i}^{j}}
\rho_{n-l}\right)  ^{-1}v_{i}\right] \\
&  =\left[  b_{0,n}+\sum_{j=1}^{k}b_{j,n}\left(  \prod_{i=1}^{j}\rho
_{n-i}\right)  ^{-1}\right]  u_{0}-\sum_{j=1}^{k}b_{j,n}\sum_{i=1}^{j}\left(
{\displaystyle\prod_{l=i}^{j}}
\rho_{n-l}\right)  ^{-1}v_{i}%
\end{align*}

Again, setting the coefficient of $u_{0}$ shows that $\{\rho_{n}\}$ satisfies
(\ref{esb}). Thus, by Lemma \ref{tdfls}, (\ref{mn}) has the linear form
symmetry if and only if the above $\{\rho_{n}\}$ satisfies (\ref{esa}) and
(\ref{esb}).

What is left in (\ref{aeq}) after all the terms with $u_{0}$ are removed is%
\[
\phi_{n}(v_{1},\ldots,v_{k})=-\sum_{j=1}^{k}\sum_{i=1}^{j}a_{j,n}\gamma
_{ij}^{-1}v_{i}+g_{n}\left(  -\sum_{j=1}^{k}\sum_{i=1}^{j}b_{j,n}\gamma
_{ij}^{-1}v_{i}\right)
\]
with $\gamma_{ij}=%
{\displaystyle\prod_{l=i}^{j}}
\rho_{n-l}$ which yields the factor equation (\ref{mnf}). The cofactor
equation (\ref{mncf}) is clear from Lemma \ref{scf}.
\end{proof}

\medskip

\begin{example}
\label{Q}To illustrate the preceding result consider the third-order
recurrence%
\begin{equation}
x_{n+1}=a_{n}x_{n}+g_{n}(x_{n}+x_{n-2}) \label{qsc}%
\end{equation}
in a left vector space $V$ over the (real) quaternions. Here $a_{n}$ is a
quaternion for each $n$ and $g_{n}:V\rightarrow V$ is an arbitrary mapping. To
examine possible sc-factorizations of (\ref{qsc}) we consider (\ref{esa}) and
(\ref{esb}) for this recurrence, where $a_{0,n}=a_{n}$, $b_{0,n}=b_{2,n}=1$
and $a_{1,n}=a_{2,n}=b_{1,n}=0$ for every $n$. We obtain%
\[
a_{n}=\alpha_{n},\quad1+\left(  \alpha_{n-1}\alpha_{n-2}\right)  ^{-1}=0
\]

This implies that the quaternions $a_{n}$ must satisfy%
\begin{equation}
a_{n}a_{n-1}+1=0 \label{p2}%
\end{equation}
for all $n\geq2$ but may otherwise be arbitrary. There are an infinite number
of possibilities, including infinitely many constants $a_{n}=a$ since the
polynomial equation $a^{2}+1=0$ has infinitely many quaternion solutions. In
particular, $a$ may be any one of the base quaternions $\pm i$, $\pm j$, $\pm
k$. There are also infinitely many non-constant solutions of type
$\{a_{n}\}=\{a,-1/a,a,-1/a,\ldots\}$ for each quaternion $a\not =0$. For all
such $a_{n}$, Theorem \ref{scmn} yields the sc-factorization of (\ref{qsc}) as%
\begin{align}
t_{n+1}  &  =g_{n}(-a_{n-1}^{-1}t_{n}-(a_{n-1}a_{n-2})^{-1}t_{n-1}%
)=g_{n}(a_{n}t_{n}+t_{n-1}),\label{sc2}\\
x_{n+1}  &  =a_{n}x_{n}+t_{n+1}\nonumber
\end{align}

\end{example}

The following corollary is an immediate consequence of Theorem \ref{scmn} for
vector spaces.

\begin{corollary}
\label{n}Let $V$ be a vector space over a field $\mathcal{F}$. If
$\{\alpha_{n}\}$ is a sequence of nonzero elements in $\mathcal{F}$\ that
satisfies (\ref{esa}) and (\ref{esb}) then (\ref{mn}) has the linear form
symmetry and the corresponding sc-factorization in $V$ into the system
(\ref{mnf})-(\ref{mncf}).
\end{corollary}

\begin{example}
\label{CxR}The recurrence (\ref{qsc}) also has sc-factorizations in complex
vector spaces, although the choice of $a_{n}$ is more restricted in the
complex field $\mathbb{C}$. In particular, rather than an infinite number of
constants, the only constant values that work now are $\pm i$.

Over a real vector space, (\ref{qsc}) has no sc-factorization relative to the
linear form symmetry with constant $a_{n}=a\in\mathbb{R}$. However, (\ref{p2})
has an infinite number of non-constant solutions in $\mathbb{R}$ of period 2,
i.e. $\{a_{n}\}=\{a,-1/a,a,-1/a,\ldots\}$ for each real $a\not =0$. Each such
solution yields a version of (\ref{qsc}) that has a sc-factorization in any
vector space over $\mathbb{R}$.
\end{example}

Each of the equalities (\ref{esa}) and (\ref{esb}) is also a difference
equation and $\{\alpha_{n}\}$ is their common solution. The existence of this
common solution establishes the proper relationship among the coefficients
$a_{i,n}$ and $b_{i,n}$ for the occurrence of the linear form symmetry, which
is independent of the functions $g_{n}.$ We now identify this relationship
explicitly in the second-order case where $k=1$.

\begin{corollary}
\label{1}Consider the following second-order recurrence in $M$:%
\begin{equation}
x_{n+1}=a_{0,n}x_{n}+a_{1,n}x_{n-1}+g_{n}(b_{0,n}x_{n}+b_{1,n}x_{n-1}).
\label{g2}%
\end{equation}

If $b_{0,n},b_{1,n}$ are units in $R$ and the following equality holds:
\begin{equation}
a_{0,n}-a_{1,n}b_{1,n}^{-1}b_{0,n}+b_{0,n+1}^{-1}b_{1,n+1}=0 \label{id}%
\end{equation}
then (\ref{g2}) has a sc-factorization into first-order recurrences in $M$ as
follows:%
\begin{align*}
t_{n+1}  &  =a_{1,n}b_{1,n}^{-1}b_{0,n}t_{n}+g_{n}\left(  b_{0,n}t_{n}\right)
\\
x_{n+1}  &  =-b_{0,n+1}^{-1}b_{1,n+1}x_{n}+t_{n+1}.
\end{align*}

\end{corollary}

\begin{proof}
A sequence of units $\{\alpha_{n}\}$ must satisfy (\ref{esa}) and (\ref{esb}),
i.e. both of the following must hold%
\[
a_{0,n}+a_{1,n}\alpha_{n-1}^{-1}=\alpha_{n},\qquad b_{0,n}+b_{1,n}\alpha
_{n-1}^{-1}=0
\]
for all $n$. If $b_{0,n},b_{1,n}$ are units then the second of the above
equalities yields $\alpha_{n-1}=-b_{0,n}^{-1}b_{1,n}$ or equivalently,
$\alpha_{n}=-b_{0,n+1}^{-1}b_{1,n+1}.$ This is also a sequence of units that
when substituted in the first equality and terms rearranged, we obtain
(\ref{id}). The sc-factorization now follows readily from Theorem \ref{scmn}.
\end{proof}

\medskip

We note that using (\ref{id}), an alternative form of the factor equation in
Corollary \ref{1} is the following:%
\[
t_{n+1}=-(a_{0,n}+b_{0,n+1}^{-1}b_{1,n+1})t_{n}+g_{n}\left(  b_{0,n}%
t_{n}\right)  .
\]

\section{Constant coefficients and polynomial roots\label{constc}}

If $a_{i,n},b_{i,n}$ are independent of $n$ then (\ref{esa}) and (\ref{esb})
may have a common fixed point (constant solution) in the group of units even
if (\ref{mn}) is still explicitly dependent on $n$ (i.e. is non-autonomous)
via $g_{n}$. This interesting and important special case is the subject of the
following consequence of Theorem \ref{scmn} that furnishes a sufficient
condition for the reducibility of (\ref{mn}) that is generally easier to check
than finding a common solution of (\ref{esa}) and (\ref{esb}).

\begin{theorem}
\label{fsor}Let $g_{n}$ be a sequence of functions on the $R$-module $M$ where
$R$ is a commutative ring with identity and let $a_{i},b_{i}\in R$ for
$i=0,1,\ldots k$. If the polynomials%
\[
P(\lambda)=\lambda^{k+1}-\sum_{i=0}^{k}a_{i}\lambda^{k-i},\quad Q(\lambda
)=\sum_{i=0}^{k}b_{i}\lambda^{k-i}%
\]
have a common root $\rho$ that is a unit in $R$ then the recurrence
\begin{equation}
x_{n+1}=\sum_{i=0}^{k}a_{i}x_{n-i}+g_{n}\left(  \sum_{i=0}^{k}b_{i}%
x_{n-i}\right)  \label{mna}%
\end{equation}
has the sc-factorization%
\begin{align}
t_{n+1}  &  =-\sum_{i=0}^{k-1}p_{i}t_{n-i}+g_{n}\left(  \sum_{i=0}^{k-1}%
q_{i}t_{n-i}\right) \label{fe}\\
x_{n+1}  &  =\rho x_{n}+t_{n+1} \label{cfe}%
\end{align}
where%
\[
p_{i}=\rho^{i+1}-a_{0}\rho^{i}-\cdots-a_{i}\quad\text{and\quad}q_{i}=b_{0}%
\rho^{i}+b_{1}\rho^{i-1}+\cdots+b_{i}%
\]

\end{theorem}

\begin{proof}
If with constant coefficients $a_{i}$ (\ref{esa}) has a constant solution
$\alpha_{n}=\alpha$ that is a unit in $R$ then $\alpha$ satisfies the equality%
\[
a_{0}+a_{1}\alpha^{-1}+a_{2}\left(  \alpha^{2}\right)  ^{-1}+\cdots
+a_{k}\left(  \alpha^{k}\right)  ^{-1}=\alpha
\]

Multiplication by $\alpha^{k}$ shows this equality to be equivalent to%
\[
\alpha^{k+1}-a_{0}\alpha^{k}-\cdots a_{k-1}\alpha-a_{k}=0
\]
i.e. $\alpha$ is a root of the polynomial $P.$ Conversely, any unit root of
$P$ is evidently a constant solution of (\ref{esa}). Similarly, $\alpha$ is a
constant solution of (\ref{esb}) with constant coefficients $b_{i}$ if and
only if it is a root of $Q.$ Therefore, a common root $\rho$ of $P$ and $Q$ is
a common, constant solution of (\ref{esa}) and (\ref{esb}) and if $\rho$ is a
unit in $R$ then (\ref{mna}) has a sc-factorization by Theorem \ref{scmn}.

Next, to find (\ref{fe}) we start with (\ref{mnf})%
\[
t_{n+1}=-\sum_{j=1}^{k}\sum_{i=1}^{j}a_{j}\left(  \gamma_{ij}\right)
^{-1}t_{n-i+1}+g_{n}\left(  -\sum_{j=1}^{k}\sum_{i=1}^{j}b_{j}\left(
\gamma_{ij}\right)  ^{-1}t_{n-i+1}\right)
\]
where for each $i=1,\cdots,k,$
\begin{align*}
\sum_{j=1}^{k}\sum_{i=1}^{j}a_{j}\left(  \gamma_{ij}\right)  ^{-1}t_{n-i+1}
&  =\sum_{j=1}^{k}\sum_{i=1}^{j}a_{j}\left(  \rho^{j-i+1}\right)
^{-1}t_{n-i+1}\\
&  =\sum_{i=1}^{k}\sum_{j=i}^{k}a_{j}\left(  \rho^{-1}\right)  ^{j-i+1}%
t_{n-i+1}\\
&  =\sum_{i=1}^{k}(a_{i}\rho^{k-i+1}+a_{i+1}\rho^{k-i}+\cdots+a_{k})(\rho
^{-1})^{k-i+1}t_{n-i+1}%
\end{align*}

Since $\rho$ is a unit root of the polynomial $P$ it follows that%
\begin{align*}
\sum_{j=1}^{k}\sum_{i=1}^{j}a_{j}\left(  \gamma_{ij}\right)  ^{-1}t_{n-i+1}
&  =\sum_{i=1}^{k}(\rho^{k+1}-a_{0}\rho^{k}-\cdots-a_{i-1}\rho^{k-i+1}%
)(\rho^{-1})^{k-i+1}t_{n-i+1}\\
&  =\sum_{i=1}^{k}p_{i-1}t_{n-i+1}%
\end{align*}
where $p_{i}$ is as defined in the statement of the corollary. The last
quantity is clearly equivalent to $\sum_{i=0}^{k-1}p_{i}t_{n-i}$ as presented
in (\ref{fe}). Similarly, since $\rho$ is also a root of the polynomial $Q$ it
follows that%
\[
\sum_{j=1}^{k}\sum_{i=1}^{j}b_{j}\left(  \gamma_{ij}\right)  ^{-1}%
t_{n-i+1}=\sum_{i=0}^{k-1}q_{i}t_{n-i+1}%
\]
where $q_{i}$ is as defined in the statement of the corollary. This completes
the derivation of (\ref{fe}); finally, (\ref{cfe}) is clear from (\ref{mncf}).
\end{proof}

\medskip

\begin{remark}
1. If the coefficients $a_{j,n}=a_{j}$ are constants for $j=0,1,\ldots,k$ and
all $n$ then (\ref{esa}) and (\ref{esb}) can be written in the equivalent
forms%
\begin{align}
\alpha_{n}\alpha_{n-1}\cdots\alpha_{n-k}-a_{0}\alpha_{n-1}\alpha_{n-2}%
\cdots\alpha_{n-k}-a_{1}\alpha_{n-2}\cdots\alpha_{n-k}-\cdots-a_{k}  &
=0\label{ce1}\\
b_{0}\alpha_{n-1}\alpha_{n-2}\cdots\alpha_{n-k}+b_{1}\alpha_{n-2}\cdots
\alpha_{n-k}+\cdots+b_{k}  &  =0 \label{ce2}%
\end{align}

These polynomial difference equations resemble the polynomials $P$ and $Q$ in
the preceding corollary in that the order of each term in them equals the
power of $\lambda$ in the corresponding term of $P$ or $Q$.

2. \noindent Note also that the factor quation (\ref{fe}) is of the same type
as (\ref{mna}) but with order reduced by one. Hence, Theorem \ref{fsor} can be
applied again to reduce the order of (\ref{fe}) further, if the associated
polynomials of (\ref{fe}) also share a common unit root. This happens when $P$
and $Q$ share more than one unit root; see the next section$.$ The theory for
repeated applications of sc-factorization on fields is discussed in
\cite{fsor}.
\end{remark}

\begin{example}
Let $M$ be an arbitrary module over the finite ring $\mathbb{Z}_{m}$ of
integers modulo $m\geq3$ and consider the third-order recurrence%
\begin{equation}
x_{n+1}=2x_{n-1}+x_{n-2}+g_{n}(x_{n}-x_{n-2}) \label{zp}%
\end{equation}
in $M.$ This recurrence is of type (\ref{mna}) with coefficients $a_{0}=0$,
$a_{1}=2$, $a_{2}=1$ and $b_{0}=1$, $b_{1}=0$, $b_{2}=-1$. These values yield
the polynomials
\[
P(\lambda)=\lambda^{3}-2\lambda-1,\qquad Q(\lambda)=\lambda^{2}-1.
\]

These share a unit root $\rho=-1$ in $\mathbb{Z}_{m}$ so Theorem \ref{fsor}
applies with $p_{0}=\rho=-1$, $p_{1}=\rho^{2}-2=-1$ and $q_{0}=1$, $q_{1}%
=\rho=-1.$ We obtain the sc-factorization%
\[
t_{n+1}=t_{n}+t_{n-1}+g_{n}(t_{n}-t_{n-1}),\qquad x_{n+1}=-x_{n}+t_{n+1}.
\]

\end{example}

The following is the common special case of Corollary \ref{1} and Theorem
\ref{fsor}.

\begin{corollary}
\label{2}Let $g_{n}$ be a sequence of functions on the $R$-module $M$ where
$R$ is a commutative ring with identity and let $a_{0},a_{1},b_{0},b_{1}$ be
fixed elements in the ring $R$. If $b_{0}$ and $b_{1}$ are units in $R$ and
the following equality holds:
\begin{equation}
a_{0}-a_{1}b_{1}^{-1}b_{0}+b_{0}^{-1}b_{1}=0 \label{c}%
\end{equation}
then the recurrence%
\begin{equation}
x_{n+1}=a_{0}x_{n}+a_{1}x_{n-1}+g_{n}(b_{0}x_{n}+b_{1}x_{n-1}) \label{g2a}%
\end{equation}
has a sc-factorization in $M$ that is given by%
\begin{align*}
t_{n+1}  &  =a_{1}b_{1}^{-1}b_{0}t_{n}+g_{n}\left(  b_{0}t_{n}\right) \\
x_{n+1}  &  =-b_{0}^{-1}b_{1}x_{n}+t_{n+1}.
\end{align*}

\end{corollary}

\begin{proof}
The statements are obviously true by Theorem \ref{fsor}. To see that they are
also true by Corollary \ref{1}, note that for (\ref{g2a})%
\[
P(\lambda)=\lambda^{2}-a_{0}\lambda-a_{1},\quad Q(\lambda)=b_{0}\lambda+b_{1}%
\]

Now $Q(\lambda)=0$ if and only if $\lambda=-b_{0}^{-1}b_{1}$ which is a unit.
Next, $P(-b_{0}^{-1}b_{1})=0$ if and only if (\ref{c}) holds and the proof is complete.
\end{proof}

\medskip

An application of the above corollary to a system of nonlinear, second-order
equations is given in Example \ref{ds} below. For linear non-homogeneous
equations in rings, Corollary \ref{fsor} can be applied even to certain types
of variable coefficients, as in the next result.

\begin{corollary}
\label{fsor1}Let $a_{i},b_{i}\in R$ for $i=0,1,\ldots k$ and $c_{n},d_{n}\in
R$ for $n\geq1$ where $R$ is a commutative ring with identity. If the
polynomials $P$ and $Q$ in Theorem \ref{fsor} have a common root $\rho$ that
is a unit $R$ then the recurrence
\begin{equation}
x_{n+1}=\sum_{i=0}^{k}(a_{i}+b_{i}c_{n})x_{n-i}+d_{n} \label{lna}%
\end{equation}
with variable coefficients has a sc-factorization in the ring $R$ into the
following pair of linear non-homogeneous recurrences:
\[
t_{n+1}=-\sum_{i=0}^{k-1}(p_{i}+q_{i}c_{n})t_{n-i}+d_{n},\quad x_{n+1}=\rho
x_{n}+t_{n+1}%
\]
where $p_{i},q_{i}$ are as defined in Theorem \ref{fsor}.
\end{corollary}

\begin{proof}
Define $g_{n}:R\rightarrow R$ as $g_{n}(r)=c_{n}r+d_{n}.$ Using these $g_{n}$
in (\ref{mna}) and re-grouping terms yields (\ref{lna}). Now apply Theorem
\ref{fsor} to conclude the proof.
\end{proof}

\section{sc-factorization of higher order systems\label{sys}}

As previously mentioned, a recurrence of order $k$ in a field may be unfolded
to a system of first-order recurrences in a $k$-dimensional vector space over
the field. Such a representation is not unique but a standard version of it is
(\ref{veq}). The reverse process where a first-order system is folded to a
higher order recurrence is also possible; see \cite{sedF} and its bibliography
for the method and its background. This folding process is useful when the
higher order equation is familiar or more tractable than the system that
yields it.

Modules lead to a different kind of folding that we discuss in this section.
Certain systems of \textit{higher order} recurrences in a direct product of a
ring with itself may be represented by higher order recurrences in modules. To
illustrate this case, consider (\ref{sed1}) in a commutative ring $R$ with
identity. Let $x_{n}=(x_{1,n},x_{2,n})$ which is in $R^{2}=R\times R.$ We
write $g_{n}:R^{2}\rightarrow R^{2}$ as%
\[
g_{n}(u,v)=(g_{n}^{1}(u,v),g_{n}^{2}(u,v))
\]
with component functions $g_{n}^{1},g_{n}^{2}:R^{2}\rightarrow R$ and obtain
the system form of (\ref{sed1}) as%
\begin{align}
x_{1,n+1}  &  =ax_{1,n}+bx_{1,n-1}+g_{n}^{1}(x_{1,n}-cx_{1,n-1},x_{2,n}%
-cx_{2,n-1})\label{nso2a}\\
x_{2,n+1}  &  =ax_{2,n}+bx_{2,n-1}+g_{n}^{2}(x_{1,n}-cx_{1,n-1},x_{2,n}%
-cx_{2,n-1}) \label{nso2b}%
\end{align}

Note that dependence of $x_{1,n+1}$ on $x_{2,n}$ or $x_{2,n-1}$ (and
similarly, of $x_{2,n+1}$ on $x_{1,n}$ or $x_{1,n-1}$) occurs via the maps
$g_{n}^{i}$ and this dependence may occur linearly as well as in a nonlinear
way. To highlight this aspect, we extract the possible linear terms from the
maps $g_{n}^{i}$ as follows:%
\[
g_{n}^{i}(u,v)=A_{i,n}u+B_{i,n}v+h_{n}^{i}(u,v),\quad i=1,2
\]
where $A_{i,n},B_{i,n}\in R$ for all $n$ and $h_{n}^{i}$ do not have linear
terms. Now we substitute into (\ref{nso2a})-(\ref{nso2b}) and rearrange terms
to obtain the equivalent system%
\begin{align}
x_{1,n+1}  &  =(a+A_{1,n})x_{1,n}+(b-cA_{1,n})x_{1,n-1}+B_{1,n}x_{2,n}%
-cB_{1,n}x_{2,n-1}+\label{gh1}\\
&  \hspace{3in}h_{n}^{1}(x_{1,n}-cx_{1,n-1},x_{2,n}-cx_{2,n-1})\nonumber\\
x_{2,n+1}  &  =A_{2,n}x_{1,n}-cA_{2,n}x_{1,n-1}+(a+B_{2,n})x_{2,n}%
+(b-cB_{2,n})x_{2,n-1}+\label{gh2}\\
&  \hspace{3in}h_{n}^{2}(x_{1,n}-cx_{1,n-1},x_{2,n}-cx_{2,n-1})\nonumber
\end{align}

In the form (\ref{gh1})-(\ref{gh2}) we see that the coefficients of the linear
terms of the system are less obvious than in the form (\ref{nso2a}%
)-(\ref{nso2b}).

We consider (\ref{sed1}) a representation of (\ref{gh1})-(\ref{gh2}), or of
(\ref{nso2a})-(\ref{nso2b}) in the direct product $R\times R$ viewed as a left
$R$-module. The results on semiconjugate factorization discussed above depend
on the coefficients $a,b,c_{n}$, etc rather than on the functions $h_{n}^{i}$
or on the modules; therefore, they apply to a broad range of systems,
including those in higher dimensions.

Once the equivalent recurrence for a system is found, we obtain the lower
order factor and cofactor equations in the module, which we can transform back
to systems. This pair of lower order systems is then considered a
sc-factorization of the original system. In the next example, we illustrate
the application of Corollary \ref{2} to a system of second-order recurrences
that can be represented in the form (\ref{sed1}).

\begin{example}
\label{ds}Let $c_{n},d_{n}$ be nonzero complex numbers for every $n$ and
consider the following system of higher order difference equations in the
vector space $\mathbb{C}^{2}$
\begin{align}
x_{1,n+1}  &  =x_{1,n}+\frac{c_{n}(x_{1,n}-x_{1,n-1})}{x_{2,n}-x_{2,n-1}%
}\label{hos1}\\
x_{2,n+1}  &  =x_{2,n}+d_{n}(x_{1,n}-x_{1,n-1}) \label{hos2}%
\end{align}

As this is also a nonlinear system, well-known methods are not available in
the literature for analyzing its solutions. In fact, given the possibility
that 0 may occur in the denominator of (\ref{hos1}) at some iteration $n$
where $x_{2,n}=x_{2,n-1}$, we need to verify the existence of solutions. The
sc-factorization of (\ref{hos1})-(\ref{hos2}) not only clarifies the existence
of its solutions, but it may also be used to calculate those solutions. In the
case of (\ref{hos1})-(\ref{hos2}), we define the functions $g_{n}%
:\mathbb{C}^{2}\rightarrow\mathbb{C}^{2}$ as
\[
g_{n}(u,v)=\left(  c_{n}\frac{u}{v},d_{n}u\right)
\]

Then, with $x_{n}=(x_{1,n},x_{2,n})$ the system (\ref{hos1})-(\ref{hos2}) can
be written as%
\begin{equation}
x_{n+1}=x_{n}+g_{n}(x_{n}-x_{n-1}) \label{hos}%
\end{equation}
which is a recurrence of type (\ref{g2a}) with $a_{0}=1$, $a_{1}=0$, $b_{0}=1$
and $b_{1}=-1$. These numbers clearly satisfy (\ref{c}) so Corollary \ref{2}
yields the following sc-factorization for (\ref{hos}):%
\begin{align}
t_{n+1}  &  =g_{n}\left(  t_{n}\right) \label{hosf}\\
x_{n+1}  &  =x_{n}+t_{n+1} \label{hosc}%
\end{align}

Note that each of the factor and cofactor equations is a first-order system;
(\ref{hosf}) is the system%
\begin{equation}
t_{1,n+1}=c_{n}\frac{t_{1,n}}{t_{2,n}},\quad t_{2,n+1}=d_{n}t_{1,n}
\label{hosf1}%
\end{equation}
that we call the factor system of (\ref{hos1})-(\ref{hos2}) and (\ref{hosc})
is%
\[
x_{1,n+1}=x_{1,n}+t_{1,n+1},\quad x_{2,n+1}=x_{2,n}+t_{2,n+1}%
\]
i.e. the cofactor system. Each of the above systems is relatively simple to
analyze. For instance, if the ratio $c_{n}/d_{n-1}=\delta$ is constant then
from (\ref{hosf1}) we obtain%
\[
t_{1,n+1}=c_{n}\frac{t_{1,n}}{t_{2,n}}=c_{n}\frac{t_{1,n}}{d_{n-1}t_{1,n-1}%
}=\frac{\delta t_{1,n}}{t_{1,n-1}}%
\]

Iterating this recurrence with initial values $t_{1,1}=x_{1,1}-x_{1,0}$,
$t_{1,2}=c_{1}t_{1,1}/t_{2,1}$ where $t_{2,1}=x_{2,1}-x_{2,0}$ and using
induction, we obtain the following solution of period 6 for it:%
\[
\{t_{1,n}\}=\left\{  t_{1,1},t_{1,2},\frac{\delta t_{1,2}}{t_{1,1}}%
,\frac{\delta^{2}}{t_{1,1}},\frac{\delta^{2}}{t_{1,2}},\frac{\delta t_{1,1}%
}{t_{1,2}},\ldots\right\}
\]

Using this, we then obtain $t_{2,n}=d_{n-1}t_{1,n-1}$ for $n\geq2$, which need
not be periodic and further,%
\[
x_{1,n}=x_{1,0}+\left[  \frac{n}{6}\right]  \left(  t_{1,1}+t_{1,2}%
+\frac{\delta t_{1,2}}{t_{1,1}}+\frac{\delta^{2}}{t_{1,1}}+\frac{\delta^{2}%
}{t_{1,2}}+\frac{\delta t_{1,1}}{t_{1,2}}\right)  +\sum_{j=1}^{r_{n}}t_{1,j}%
\]
where $[m]$ is the greatest integer less than or equal to $m$ and $r_{n}$ is
the remainder of the fraction $n/6$ (the sum is dropped if $r_{n}=0$). We
calculate $x_{2,n}$ similarly to obtain the complete solution of the system
(\ref{hos1})-(\ref{hos2}).

The above calculations show, in particular, that (\ref{hos1})-(\ref{hos2}) has
solutions as long as $t_{1,1},t_{1,2}\not =0.$ Since
\[
t_{1,2}=c_{1}\frac{t_{1,1}}{t_{2,1}}=\frac{c_{1}(x_{1,1}-x_{1,0})}%
{x_{2,1}-x_{2,0}}%
\]
the existence of solutions is established if the initial values satisfy
$x_{1,1}\not =x_{1,0}$ and $x_{2,1}\not =x_{2,0}.$ A little more calculation
shows that this conclusion is valid even for the non-autonomous case where
$c_{n}/d_{n-1}$ is not constant and thus $\{t_{1,n}\}$ may not be periodic.
\end{example}

The next example discusses a variation of the system (\ref{hos1})-(\ref{hos2}).

\begin{example}
Consider the following system in the vector space $\mathbb{C}^{2}%
=\mathbb{C}\times\mathbb{C}$:
\begin{align}
x_{1,n+1}  &  =x_{1,n}+\frac{c_{n}(x_{1,n}-x_{1,n-1})}{x_{2,n}-x_{2,n-1}%
}\label{m1}\\
x_{2,n+1}  &  =x_{2,n-1}+d_{n}(x_{1,n}-x_{1,n-1}) \label{m2}%
\end{align}

For this system we use the form (\ref{gh1})-(\ref{gh2}) with $c=1$ and
parameters defined via the following equations for all $n$%
\begin{align*}
a+A_{1,n}  &  =1,\quad b-A_{1,n}=0,\quad B_{1,n}=0\\
a+B_{2,n}  &  =0,\quad b-B_{2,n}=1,\quad A_{2,n}=d_{n}%
\end{align*}

From the above equations we obtain
\begin{equation}
A_{1,n}=b=1-a,\quad B_{2,n}=-a,\quad A_{2,n}=d_{n},\quad B_{1,n}=0 \label{ps}%
\end{equation}
for all $n$ for arbitrary $a\in R.$ These values define the functions%
\[
g_{n}(u,v)=\left(  (1-a)u+\frac{c_{n}u}{v},d_{n}u-av\right)
\]
and the system (\ref{m1})-(\ref{m2}) is represented by the recurrence%
\begin{equation}
x_{n+1}=ax_{n}+(1-a)x_{n-1}+g_{n}(x_{n}-x_{n-1}) \label{mrep}%
\end{equation}

The coefficients in the above equations satisfy (\ref{c}) with $a_{0}=a$,
$a_{1}=1-a$, $b_{0}=1$ and $b_{1}=-1$ so Corollary \ref{2} yields the
sc-factorization%
\begin{align*}
t_{n+1}  &  =g_{n}\left(  t_{n}\right)  -(1-a)t_{n}\\
x_{n+1}  &  =x_{n}+t_{n+1}.
\end{align*}

The special values $a=0,1$ simplify the various expressions. With $a=1$ we
obtain
\[
g_{n}(u,v)=\left(  \frac{c_{n}u}{v},d_{n}u-v\right)
\]
and the factor equation in the sc-factorization reduces to $t_{n+1}%
=g_{n}\left(  t_{n}\right)  .$ Note also that in this case, (\ref{mrep}) is
the same as (\ref{hos}). With $a=0$ we have%
\[
g_{n}(u,v)=\left(  u+\frac{c_{n}u}{v},d_{n}u\right)  =u\left(  1+\frac{c_{n}%
}{v},d_{n}\right)
\]
and the factor equation is $t_{n+1}=g_{n}\left(  t_{n}\right)  -t_{n}$. In
this case, (\ref{mrep}) is not the same as (\ref{hos}); rather $x_{n+1}%
=x_{n-1}+g_{n}(x_{n}-x_{n-1})$ is the representing recurrence on
$\mathbb{C}^{2}$ now.
\end{example}

Note that the preceding example also shows that representations of systems in
rings by equations in modules (and thus also the corresponding
sc-factroizations) are \textit{not unique} in general.

The ideas in the examples of this section extend to systems of recurrences of
arbitrary order in vector spaces of any dimension (including infinite, as in
Example \ref{nrmo3} below) as long as it is possible to transform the system
to an equation of type (\ref{mna}) in the vector space. We emphasize that this
transformation is not possible in general and leave as an open problem a
classification of systems that can be represented by recurrences in modules.

\section{Polynomial roots and repeated sc-factorization\label{repfac}}

An important by-product of Theorem \ref{fsor} is that the factor equation of
(\ref{mna}), namely, (\ref{fe}) is of the same type as (\ref{mna}), but has a
lower order. This feature, which is also one of the advantages that the linear
form symmetry enjoys compared with other form symmetries, suggests that
Theorem \ref{fsor} can be applied repeatedly, each time reducing the order of
the original recurrence by one, as long as common, unit polynomial roots exist.

The next result simplifies the task of searching for shared polynomial roots
by limiting it to all common roots of the \textit{original} pair of
polynomials $P$ and $Q.$

\begin{lemma}
\label{replem}Let $R$ be an integral domain and let $\rho\in R$ be a common
unit root of $P$ and $Q$. If $\rho_{1}\in R$ is another shared unit root of
$P$ and $Q$ in $R$ then $\rho_{1}$ is also a common unit root of both of the
following polynomials%
\begin{align*}
P_{1}(\lambda)  &  =\lambda^{k}+p_{0}\lambda^{k-1}+p_{1}\lambda^{k-2}%
+\cdots+p_{k-1}\\
Q_{1}(\lambda)  &  =q_{0}\lambda^{k-1}+q_{1}\lambda^{k-2}+\cdots+q_{k-1}%
\end{align*}
where the coefficients $p_{i}$ and $q_{i}$ are defined in Theorem \ref{fsor}.
\end{lemma}

\begin{proof}
By assumption,%
\begin{equation}
P(\rho_{1})=Q(\rho_{1})=0. \label{pq}%
\end{equation}

Now%
\begin{align*}
(\lambda-\rho)P_{1}(\lambda)  &  =(\lambda-\rho)\left(  \lambda^{k}+\sum
_{j=0}^{k-1}p_{\,j}\lambda^{k-j-1}\right) \\
&  =\lambda^{k+1}+\sum_{j=0}^{k-1}[p_{\,j}-\rho p_{\,j-1}]\lambda^{k-j}-\rho
p_{k-1}.
\end{align*}
where we define $p_{-1}=1$ to simplify the notation. For each $j=0,1,\ldots
,k-1$ note that%
\[
p_{j}-\rho p_{j-1}=-a_{j}%
\]
and further, since $P(\rho)=0$,%
\begin{align*}
\rho p_{k-1}  &  =\rho\left(  \rho^{k}-a_{0}\rho^{k-1}-\cdots-a_{k-1}\right)
\\
&  =P(\rho)+a_{k}\\
&  =a_{k}.
\end{align*}

Thus
\[
(\lambda-\rho)P_{1}(\lambda)=P(\lambda)
\]
and if $\rho_{1}\not =\rho$ then $P_{1}(\rho_{1})=0$ by (\ref{pq}). If
$\rho_{1}=\rho$ then $\rho$ is a double root of both $P$ and $Q$ so that their
derivatives are zeros, i.e.,%
\begin{equation}
P^{\prime}(\rho)=Q^{\prime}(\rho)=0. \label{pqder}%
\end{equation}

Therefore,
\begin{align*}
P_{1}(\rho)  &  =\rho^{k}+\sum_{j=0}^{k-1}(\rho^{j+1}-a_{0}\rho^{j}%
-\cdots-a_{j-1}\rho-a_{j})\rho^{k-j-1}\\
&  =(k+1)\rho^{k}-\sum_{j=0}^{k-1}(k-j)a_{j}\rho^{k-j-1}\\
&  =P^{\prime}(\rho).
\end{align*}

In particular, if $\rho_{1}=\rho$ then $P_{1}(\rho_{1})=0$ by (\ref{pqder}).
Similar calculations show that $Q_{1}(\rho_{1})=0$, thus completing the proof.
\end{proof}

\medskip

Theorem \ref{fsor} and Lemma \ref{replem} imply the following result.

\begin{corollary}
\label{repalor}(a) Assume that the polynomials $P$ and $Q$ have two common,
unit roots $\rho,\rho_{1}$ in an integral domain $R$. Then the factor equation
(\ref{fe}) of (\ref{mna}) has the linear form symmetry in the $R$-module $M$
with the factor equation%
\begin{equation}
r_{n+1}=-\sum_{i=0}^{k-2}p_{1,i}r_{n-i}+g_{n}\left(  \sum_{i=0}^{k-2}%
q_{1,i}r_{n-i}\right)  \label{alfac2}%
\end{equation}
where%
\[
p_{1,i}=\rho_{1}^{i+1}+p_{0}\rho_{1}^{i}+\cdots+p_{i}\text{ and }q_{1,i}%
=q_{0}\rho_{1}^{i}+q_{1}\rho_{1}^{i-1}+\cdots+q_{i}.
\]

(b) The recurrence (\ref{mna}) has a secondary or repeated SC factorization
that consists of the factor equation (\ref{alfac2}) and the two cofactor
equations%
\begin{align*}
t_{n+1}  &  =\rho_{1}t_{n}+r_{n+1}\\
x_{n+1}  &  =\rho x_{n}+t_{n+1}.
\end{align*}

(c) If $P$ and $Q$ have $m$ common, unit roots in $R$ (counting repeated or
multiple roots) where $1\leq m\leq k$ then (\ref{mna}) can be reduced in order
repeatedly $m$ times in $M.$
\end{corollary}

\begin{proof}
(a) The proof applies the same arguments as in the proof of Theorem \ref{fsor}
to the factor equation (\ref{fe}) that was obtained using the shared unit root
$\rho.$ We simply change $a_{i}$ to $-p_{i}$ and $b_{i}$ to $q_{i}$ in the
proof and make other minor modifications to complete the proof here.

(b) This follows from the general form of cofactors associated with the linear
form symmetry.

(c) This follows by induction, using Part (a).
\end{proof}

\medskip

\begin{definition}
If $m=k+1$ in Corollary \ref{repalor}(c) then we say that (\ref{mna}) has a
complete sc-factorization into a triangular systems of first-order recurrences.
\end{definition}

The next example illustrates this concept; see Corollary \ref{calsp} below for
a more general class of recurrences of the same type.

\begin{example}
\label{exzp}Let $M$ be a vector space over the finite field $\mathbb{Z}_{p}$
where $p$ is an odd prime and consider the third-order recurrence%
\begin{equation}
x_{n+1}=2x_{n-1}+x_{n-2}+g_{n}(x_{n}-x_{n-1}-x_{n-2}) \label{zp1}%
\end{equation}
in $M.$ This recurrence is of type (\ref{mna}) with coefficients $a_{0}=0$,
$a_{1}=2$, $a_{2}=1$ and $b_{0}=1$, $b_{1}=b_{2}=-1.$ These values yield the
polynomials
\[
P(\lambda)=\lambda^{3}-2\lambda-1,\qquad Q(\lambda)=\lambda^{2}-\lambda-1
\]

Note that $P(\lambda)=(\lambda+1)(\lambda^{2}-\lambda-1)=$ $(\lambda
+1)Q(\lambda).$ It is known (see, e.g. \cite{gup}) that $Q$ has two (nonzero)
roots $\rho_{1},\rho_{2}\in$ $\mathbb{Z}_{p}$ if and only if $p\equiv
0,1,4(\operatorname{mod}5).$ For instance, in $\mathbb{Z}_{5}$, 3 is a
double-root of $Q$ so $\rho_{1}=\rho_{2}=3;$ in $\mathbb{Z}_{11}$, 4 and 8 are
distinct roots. For such primes, Corollary \ref{repalor} implies that
(\ref{zp1}) has at least two sc-factorizations; since (\ref{zp1}) has order 3,
it must have a complete sc-factorization. First, using $\rho_{1}$ we obtain%
\begin{align}
t_{n+1}  &  =-\rho_{1}t_{n}-(\rho_{1}-1)t_{n-1}+g_{n}(t_{n}-(\rho
_{1}-1)t_{n-1})\nonumber\\
x_{n+1}  &  =\rho_{1}x_{n}+t_{n+1} \label{czp}%
\end{align}

Next, using $\rho_{2}$ we find the sc-factorization of the factor equation
above
\[
r_{n+1}=-r_{n}+g_{n}(r_{n}),\qquad t_{n+1}=\rho_{2}t_{n}+r_{n+1}%
\]

These equations, together with the cofactor equation (\ref{czp}) constitute
the complete sc-factorization of (\ref{zp1}) into a triangular system of three
first-order recurrences.
\end{example}

In the next example, we find the complete sc-factorization for a third-order
recurrence\ in two different ways.

\begin{example}
\label{nrmo3}Consider the following third-order functional recurrence
\begin{equation}
x_{n+1}(r)=x_{n-1}(r)+\int_{0}^{r}\phi_{n}(\tau,x_{n}-x_{n-2})(\tau
)d\tau,\quad0\leq r\leq1\label{g2c}%
\end{equation}
in the normed vector space $C[0,1]$ of continuous, complex-valued functions on
the interval $[0,1]$ and $\phi_{n}:[0,1]\times C[0,1]\rightarrow C[0,1]$ for
all $n.$ We may define
\[
g_{n}(x)(r)=\int_{0}^{r}\phi_{n}(\tau,x)d\tau
\]
and assume that the initial functions $x_{0}(r),x_{1}(r),x_{2}(r)$ are given
for $0\leq r\leq1$. The polynomials corresponding to (\ref{g2c}) are%
\[
P(\lambda)=\lambda^{3}-\lambda,\quad Q(\lambda)=\lambda^{2}-1
\]

These share \textit{two} unit roots $\pm1$ in $\mathbb{C}$ so we may apply
Corollary \ref{repalor}. First, with $\lambda=1$ we obtain the
sc-factorization%
\begin{align}
t_{n+1}(r)  &  =-t_{n}(r)+\int_{0}^{r}\phi_{n}(\tau,t_{n}+t_{n-1})(\tau
)d\tau,\label{g2cf}\\
x_{n+1}(r)  &  =x_{n}(r)+t_{n+1}(r) \label{g2cc}%
\end{align}

Next, with $\lambda=-1$ the factor equation (\ref{g2cf}) has the
sc-factorization%
\begin{align*}
s_{n+1}(r)  &  =\int_{0}^{r}\phi_{n}(\tau,s_{n})(\tau)d\tau\\
t_{n+1}(r)  &  =-t_{n}(r)+s_{n+1}(r).
\end{align*}

These equations together with (\ref{g2cc}) constitute the \textit{complete
sc-factorization} of (\ref{g2c}) into first-order recurrences. It is worth
mentioning here that a complete sc-factorization for the recurrence
(\ref{g2c}) may alternatively be obtained as follows. First, write (\ref{g2c})
as
\[
x_{n+1}(r)-x_{n-1}(r)=\int_{0}^{r}\phi_{n}(\tau,x_{n}-x_{n-2})(\tau)d\tau
\]

Using this suggestive form, we substitute $s_{n}=x_{n}-x_{n-2}$ to obtain the
first-order recurrence%
\[
s_{n+1}(r)=\int_{0}^{r}\phi_{n}(\tau,s_{n})(\tau)d\tau
\]

Note that the cofactor equation, obtained from the above substitution, now has
order 2:%
\begin{equation}
x_{n+1}(r)=x_{n-1}(r)+s_{n+1}(r). \label{cf2}%
\end{equation}

The last two equations yield an sc-factorization of (\ref{g2c}) that is of a
different type than what we obtained using Theorem \ref{fsor}; see Chapter 6
in \cite{fsor}. The cofactor (\ref{cf2}) is easily reduced using Theorem
\ref{fsor} (now with $Q(\lambda)=0$) as%
\[
t_{n+1}(r)=-t_{n}(r)+s_{n+1}(r),\qquad x_{n+1}(r)=x_{n}(r)+t_{n+1}(r)
\]

As might be expected, this is the same complete sc-factorization that we
obtained via Corollary \ref{repalor}.
\end{example}

Corollary \ref{repalor} and the above two examples raise a natural question:
Are there special cases of (\ref{mna}) that have complete sc-factorizations
into a triangular system of $k+1$ first-order recurrences? The following
result indicates two such cases; Examples \ref{exzp} and \ref{nrmo3} above
illustrate each of these cases.

\begin{corollary}
\label{calsp}Let $R$ be an integral domain and $M$ an $R$-module.

(a) Let $r,b_{j}\in R$, $j=1,\ldots,k$, $b_{k}\not =0$ and define $b_{0}=1$,
$b_{k+1}=0$. If the polynomial equation
\begin{equation}
\lambda^{k}+b_{1}\lambda^{k-1}+b_{2}\lambda^{k-2}+\cdots+b_{k}=0 \label{q}%
\end{equation}
has $k$ unit roots in $R$ then the recurrence%
\begin{equation}
x_{n+1}=\sum_{j=0}^{k}(rb_{j}-b_{j+1})x_{n-j}+g_{n}\left(  \sum_{j=0}^{k}%
b_{j}x_{n-j}\right)  \label{fsc}%
\end{equation}
has a complete sc-factorization into a triangular system of $k+1$ first-order
recurrences in $M.$

(b) Let $b,a_{j}\in R$, $j=0,\ldots,k-1$ and $a_{k-1}\not =0$. If the
polynomial equation
\begin{equation}
\lambda^{k}-a_{0}\lambda^{k-1}-a_{1}\lambda^{k-2}-\cdots a_{k-1}=0
\label{1pol}%
\end{equation}
has $k$ unit roots in $R$ then the recurrence
\begin{equation}
x_{n+1}=\sum_{j=0}^{k-1}a_{j}x_{n-j}+g_{n}\left(  bx_{n}-\sum_{j=1}^{k}%
a_{j-1}bx_{n-j}\right)  \label{alsp}%
\end{equation}
has a complete sc-factorization into a triangular system of $k+1$ first-order
recurrences in $M$.
\end{corollary}

\begin{proof}
(a) We observe that the polynomial in (\ref{q}) is $Q(\lambda)$ in this case
and from (\ref{fsc}), $P(\lambda)$ is%
\[
\lambda^{k+1}-\sum_{j=0}^{k}(rb_{j}-b_{j+1})\lambda^{k-j}=(\lambda
-r)Q(\lambda)
\]

Thus $P$ and $Q$ share $k$ unit roots in $R$ whenever $Q$ does and the proof
is complete.

(b) Let $P(\rho)=Q(\rho)=0$, i.e.%
\begin{align}
\rho^{k+1}-a_{0}\rho^{k}-a_{1}\rho^{k-1}-\cdots-a_{k-1}\rho &  =0
\label{alorp}\\
b\rho^{k}-a_{0}b\rho^{k-1}-\cdots-a_{k-2}b\rho-a_{k-1}b  &  =0. \label{alorq}%
\end{align}

If $b\not =0$ then since also $\rho,a_{k-1}\not =0$, after cancelling $\rho$
from (\ref{alorp}) and $b$ from (\ref{alorq}), both equations reduce to the
polynomial in (\ref{1pol}). Thus, every common, unit root of $P$ and $Q$ is a
zero of the polynomial equation (\ref{1pol}) and conversely, every zero of
(\ref{1pol}) is a root of both $P$ and $Q.$ Therefore, (\ref{alsp}) has $k$
sc-factorizations by Corollary \ref{repalor}.

If $b=0$ then (\ref{alorq}) is true trivially for all $\rho\in R.$ Therefore,
again the single equation (\ref{1pol}) remains and by Corollary \ref{repalor},
(\ref{alsp}) has $k$ repeated sc-factorizations. Note that in this case,
(\ref{alsp}) reduces to a linear non-homogeneous equation of order $k$.
\end{proof}

\medskip

\begin{remark}
\label{A}The recurrence (\ref{alsp}) has a sc-factorization other than the one
in Corollary \ref{calsp}(b). If $b\not =0$ then the substitution%
\[
s_{n+1}=x_{n}-\sum_{j=1}^{k}a_{j-1}x_{n-j}%
\]
yields the sc-factorization
\begin{align}
s_{n+1}  &  =g_{n}(bs_{n})\label{cc1}\\
x_{n+1}  &  =\sum_{j=1}^{k}a_{j-1}x_{n-j+1}+s_{n+1} \label{cc2}%
\end{align}

Here the factor equation has order 1 and the cofactor is a linear
non-homogeneous equation of order $k.$ For more details on this type of
sc-factorization and the corresponding form symmetry we refer to \cite{fsor}
(also see Example \ref{nrmo3}).
\end{remark}

We close this section with the following result on linear recurrences in
fields that shows, in particular, if $R$ is an algebraically closed field then
(\ref{cc2}) has a complete sc-factorization in every vector space over $R$.
Recall that if $g_{n}$ is a sequence of constants then (\ref{mna}) reduces to
a non-homogeneous linear equation. In this case, it is no loss of generality
to set $b_{j}=a_{j}$ for $j=0,1,\ldots,k,$ so that $Q=P.$ Thus the proof of
the following is clear.

\begin{corollary}
\label{linf} Let $\mathcal{F}$ be a nontrivial field, $a_{j},c_{n}%
\in\mathcal{F}$ for $j=0,1,\ldots,k$ and all $n\geq k$ and assume that
$a_{k}\not =0$. The linear non-homogeneous recurrence%
\[
x_{n+1}=\sum_{i=0}^{k}a_{i}x_{n-i}+c_{n}%
\]
has a complete sc-factorization in $\mathcal{F}$ into a triangular system of
$k+1$ first-order, linear non-homogeneous recurrences if the polynomial%
\[
P(\lambda)=\lambda^{k+1}-\sum_{i=0}^{k}a_{i}\lambda^{k-i}%
\]
factors completely in $\mathcal{F}$. In particular, every linear
non-homogeneous recurrence in an algebraically closed field has a complete sc-factorization.
\end{corollary}

\section{On the reducibility of recurrences\label{red}}

The preceding two sections show a close relationship between polynomial
factorization and the sc-factorization of (\ref{mna}). Specifically, the
existence of common unit roots for the polynomials $P$ and $Q$ in a
commutative ring $R$ with identity is sufficient for the sc-factorization of
(\ref{mna}) in any $R$-module $M$. In general, we define reducibility for
recurrences in modules as follows.

\begin{definition}
(Reducibility) The recurrence (\ref{mn}) in a left unitary $R$-module is
reducible, relative to the linear form symmetry, if the difference equations
(\ref{esa}) and (\ref{esb}) have a common solution in the unit of groups of
the ring $R$. If such a common unitary solution does not exist then (\ref{mn})
is irreducible (relative to the linear form symmetry).
\end{definition}

Note that the concept of reducibility for recurrences can only be relative to
a particular form symmetry because a recurrence may have sc-factorizations
relative to some form symmetries but not others; see \cite{fsor}.

While the existence of common roots for $P$ and $Q$ is sufficient for the
reducibility of (\ref{mna}) in every $R$-module over a commutative ring $R$
with identity, this condition is not necessary because sc-factorizations may
exist even when $P$ or $Q$ have no roots in $R$ at all. We illustrate this
possibility with an example.

\begin{example}
\label{np}Consider the recurrence%
\begin{equation}
x_{n+1}=-x_{n-1}+g_{n}(x_{n}+x_{n+2}) \label{rnp}%
\end{equation}
in an arbitrary vector space $V$ over the field $\mathbb{R}$ of real numbers
with $g_{n}:V\rightarrow V$ arbitrary functions. In this case, $P(\lambda
)=\lambda^{3}+\lambda$ has a single unit root $-1$ in $\mathbb{R}$ but
$Q(\lambda)=\lambda^{2}+1$ has no roots in $\mathbb{R}$; in particular, $P$
and $Q$ share no unit roots in $\mathbb{R}$. However, (\ref{rnp}) does have a
complete sc-factorization in vector spaces over $\mathbb{R}$ which can be
determined using Corollary \ref{n}. We seek a nonzero sequence $\alpha_{n}$ in
$\mathbb{R}$ that satisfies the two equations%
\begin{align*}
\alpha_{n}  &  =-\alpha_{n-1}^{-1}=-\frac{1}{\alpha_{n-1}},\\
0  &  =1+\left(  \alpha_{n-1}\alpha_{n-2}\right)  ^{-1}\quad\text{or\quad
}\alpha_{n-1}=-\frac{1}{\alpha_{n-2}}%
\end{align*}

These are identical difference equations, so any nonzero solution, e.g.
$\alpha_{n}=(-1)^{n}$ works for $n\geq2$. The resulting sc-factorization is%
\begin{align}
t_{n+1}  &  =(-1)^{n}t_{n}+g_{n}(t_{n}+(-1)^{n-1}t_{n-1})\label{fn}\\
x_{n+1}  &  =(-1)^{n}x_{n}+t_{n+1} \label{cn}%
\end{align}

We may now apply either Corollary \ref{n} or Corollary \ref{1} to (\ref{fn}).
Using the latter, it is easy to see that the coefficients of (\ref{fn})
satisfy (\ref{id}) and we obtain the sc-factorization%
\begin{align*}
s_{n+1}  &  =g_{n}\left(  s_{n}\right) \\
t_{n+1}  &  =(-1)^{n-1}t_{n}+s_{n+1}%
\end{align*}

This pair of first-order recurrences, together with (\ref{cn}) yield a
complete sc-factorization of (\ref{rnp}) in any vector space over $\mathbb{R}$.
\end{example}

We note that (\ref{rnp}) is a special case of (\ref{alsp}) so it also has the
alternative sc-factorization mentioned in Remark \ref{A}. This factorization
does not use of polynomials, yet it contains only constant coefficients in its
factor equation.

The difference equations (\ref{ce1}) and (\ref{ce2}) reduce to the polynomial
equations $P(\lambda)=0$ and $Q(\lambda)=0$ in Theorem \ref{fsor} when they
have a common \textit{constant} solution, or fixed point, in $R$. The question
as to whether (\ref{ce1}) and (\ref{ce2}) can have a common solution (a
sequence of units) if $P$ and $Q$ do \textit{not }have common unit roots is
presently difficult to answer in general.

A deeper understanding of (\ref{ce1}) and (\ref{ce2}), which are polynomial
difference equations, is needed that goes beyond polynomial factorization. A
starting point for this more complicated problem is linear non-homogeneous
recurrences for which (\ref{ce2}) is true trivially so we may focus only on
exploring the existence of unitary solutions for (\ref{ce1}). A preliminary
work in this direction is \cite{sedR}.

In special cases where common constant solutions of (\ref{ce1}) and
(\ref{ce2}) must exist, the polynomials $P$ and $Q$ play a more decisive role
in the study of reducibility of recurrences. We encountered one such case in
Corollary \ref{linf} for linear recurrences. To discuss a nonlinear case,
consider the recurrence
\begin{equation}
x_{n+1}=\sum_{i=0}^{k}a_{i}x_{n-i}+g_{n}(x_{n-j}-bx_{n-j-1}) \label{o2b}%
\end{equation}
where $0\leq j\leq k-1$ is fixed and $a_{i}$ is in a commutative ring with
identity for each $i=0,1,\ldots,k$. For this recurrence, (\ref{ce2}) reduces
to
\[
\alpha_{n-j-1}\alpha_{n-j-2}\cdots\alpha_{n-k}-b\alpha_{n-j-2}\cdots
\alpha_{n-k}=0
\]
if $j<k-1$ and to $\alpha_{n-k}-b=0$ if $j=k-1$, for all $n\geq k.$ Thus, only
a constant sequence $\alpha_{n}=b$ is possible and since $\alpha_{n}$ is a
unit by definition, it follows that $b$ must be a unit. With this choice,
(\ref{ce1}) reduces to the equation $P(b)=0.$ Therefore, (\ref{o2b}) is
reducible and we have the following immediate consequence of Theorem
\ref{scmn}.

\begin{corollary}
Let $R$ be a commutative ring with identity, $a_{i},b\in R$ and let
$g_{n}:M\rightarrow M$ be a sequence of functions on an $R$-module $M$. The
recurrence (\ref{o2b}) is reducible in $M$ relative to the linear form
symmetry if and only if $b$ is a unit and a root of $P$ in $R$.
\end{corollary}

From the preceding result, it readily follows that the recurrence%
\[
x_{n+1}=\sum_{i=0}^{k}a_{i}x_{n-i}+g_{n}(x_{n-j}-x_{n-j-1})
\]
is reducible relative to the linear form symmetry if and only if
\[
\sum_{i=0}^{k}a_{i}=1
\]
and similarly,
\[
x_{n+1}=\sum_{i=0}^{k}a_{i}x_{n-i}+g_{n}(x_{n-j}+x_{n-j-1})
\]
is reducible relative to the linear form symmetry if and only if%
\[
\sum_{i=0}^{k}(-1)^{k-i}a_{i}=(-1)^{k+1}.
\]

These two statements are also valid for arbitrary $\mathbb{Z}$-modules. The
results in this section identify some of the boundary of relevance for the
linear form symmetry as far as the sc-factorization of (\ref{mn}) is concerned.

\section{Concluding remarks}

While the important role that algebra in general, and polynomials in
particular play in the factorization of \textit{linear} recurrences is
well-known, the usual methods do not apply to nonlinear recurrences. As the
latter are increasingly used in scientific models, it is necessary to develop
methods that may be relevant to them. To this end, we presented an algebraic
method, namely semiconjugate factorization relative to the linear form
symmetry, that can be used to also obtain sc-factorizations of certain
nonlinear recurrences in modules. For linear recurrences, this method yields
the same results as the standard linear theory.

The preceding study answers a few questions but it leads to many more
unanswered questions, some of which are noted in the above discussion. For
instance, we noted that the ideas in the examples of Section \ref{sys} extend
to systems of recurrences of arbitrary order in vector spaces of any dimension
including infinite, as in Example \ref{nrmo3}. This can be achieved as long as
it is possible to transform the system to an equation of type (\ref{mna}) in
the vector space. However, this transformation is not possible in general a
classification of systems that can be represented by recurrences in modules is
an open problem.

Finally, a question that goes beyond the boundaries of this paper for possible
future study is what types of nonlinear recurrences other than (\ref{mn}) can
be fruitfully studied using the linear or other types of form symmetry and the
corresponding sc-factorizations. For preliminary ideas that may be useful in
such future studies we refer to \cite{fsor}.

\end{document}